\documentclass[a4]{article}

\usepackage{amsthm,amsmath,bm,amssymb,extarrows,mathptmx,bbm,wrapfig,cases,color}
\usepackage[dvipdfmx]{graphicx}
\usepackage{pxfonts}

\newtheorem{theo}{Theorem}
\newtheorem{cond}{Condition}

\newtheorem{lemm}{Lemma}
\newtheorem{rema}{Remark}

\makeatletter

\@addtoreset{equation}{section}

\@addtoreset{figure}{section}

\@addtoreset{table}{section}
\makeatother

\title{A New { D}iscretization {S}cheme for One Dimensional Stochastic Differential Equations Using Time Change Method}
\author{Masaaki Fukasawa\footnote{
		Graduate School of Engineering Science,
		Osaka University,
		1-3, Machikaneyama-cho, Toyonaka,
		Osaka, 560-8531, Japan,
		email~:~\texttt{fukasawa@sigmath.es.osaka-u.ac.jp}
	}
\quad
and
\quad
 Mitsumasa Ikeda\footnote{
 	Graduate School of Engineering Science,
 	Osaka University,
 	1-3, Machikaneyama-cho, Toyonaka,
 	Osaka, 560-8531, Japan,
 	email~:~\texttt{mikeda@sigmath.es.osaka-u.ac.jp}
 }
}
\begin{document}
\maketitle

	\begin{abstract}
		We propose a new numerical method for one dimensional
	 stochastic differential equations (SDEs). The main idea of this
	 method is based on { a} representation of
	 { a} weak
	 solution { of a SDE}
	 with a time changed Brownian motion, dated back to
	 Doeblin (1940). In { cases} where the diffusion coefficient is
	 bounded and { $\beta$-H\"{o}lder continuous with $0
	 < \beta \leq 1$}, we provide the rate of
	 strong convergence.
{An advantage of our approach is that we approximate the weak solution,
	 which enables us to treat a SDE with no strong solution.
Our scheme is the first to achieve the strong convergence for the case
	  $0 < \beta
	 < 1/2$.}
	\end{abstract}

	\section{Introduction}
	In this article, we provide a numerical method { of
	approximating} a weak solution { of} a one dimensional stochastic
	differential equation. There are many studies about numerical
	approximation for SDEs which converges strongly to the
	solution. A variety of applications { includes
	path-dependent} option pricing in
	{ financial engineering}. 
{ Here we focus on the}
following one-dimensional SDE;
	\begin{align}\label{main}
		dX_t=&\sigma(t,X_t)dW_t.
	\end{align}
{
Such a SDE model is called a {\it local volatility model}
and popular in financial practice. Although (\ref{main}) does not
include a drift term, we remark that under appropriate conditions,
a general one dimensional SDE with drift can be reduced to (\ref{main});}
 time homogeneous one-dimensional SDEs can be
	transformed to ones without drift term by {\it scale function}
	{ in a pathwise sense},
	and time inhomogeneous SDEs { also} can be transformed to martingales
	using { the Girsanov-Maruyama} transformation
	{ in a sense of law.} 

	In order to study numerical scheme of SDEs (\ref{main}), we must
	discuss the conditions where the existence and uniqueness of the
	solution {hold} in some different senses; {\it strong
	uniqueness}, {\it pathwise uniqueness} and {\it uniqueness in
	the sense of probability law}. Many researchers have been
	studied  the unique existence of the solution to SDEs for a long
	time. The most famous condition for the strong unique existence
	of the solution to SDE is Lipschitz continuity and linear growth
	of the drift and diffusion coefficient (see\cite{KandS}).

	Bru and Yor discussed in \cite{bruyo2002} about this
	issue. According to \cite{bruyo2002}, W.~Doeblin has written a
	paper about this issue before the many facts about the structure
	of martingale were found. He { showed that a} diffusion process can
	be represented by some stochastic process which is {
	driven by a } time
	changed Brownian motion. 
{ Although this Doeblin's work in 1940 became public only
	after 2000,  the idea was rediscovered and extended in
	stochastic calculus; already in a textbook~\cite{watanabe} by
	Ikeda and Watanabe in 1984, it is shown that 
%according to \cite{watanabe}, if 
a one dimensional SDE of the form (\ref{main})  in a certain
	class  has a unique solution represented as a time
	changed Brownian motion, where the time change is given as 
	a solution of a random  ordinary differential equation, 
as we see in the next section in more detail.
	We use this representation to construct a new approximation
	scheme for one
	dimensional SDEs. 
For time homogeneous case, i.e., $\sigma(t,x)=\sigma(x)$ in (\ref{main}),
Engelbert and Schmidt \cite{engelbert} gave an
	an equivalent condition for the weak existence and uniqueness
	in the sense of probability law, under which the
	weak solution is represented as the time-changed Brownian motion.
 For time homogeneous SDEs, an
	excellent survey \cite{taguchiphd}
 about the existence and uniqueness of SDEs is available.
}
	
	The most famous numerical scheme for SDEs is {\it the Euler-{Maruyama}
	Scheme}. This method approximates a solution of SDEs in a very
	similar way to {the Euler scheme} for ordinary differential
	equations. 
It is well known that the Euler-{Maruyama
	approximation}  converges to the strong solution of SDE uniformly in the sense of
	$L^p$ with convergence rate $n^{-1/2}$ when the diffusion
	coefficient is Lipschitz continuous~\cite{KandP}.
Under the $\beta$-H\"{o}lder continuity of the diffusion coefficient
 $\sigma(t,x)$,  where
	 $1/2\leq\beta\leq1$, Gy\"{o}ngy and R\'{a}sonyi \cite{gyongy2002}
{ showed}
that for any $T>0$ %and $n\geq2$
	 there exists a constant $C>0$ such that

	\begin{align}
	E\sup_{0\leq t\leq T}|X_t-X^{(n)}_t|\leq\begin{cases}\label{p1}
	\frac{C}{\mathrm{ln}^{1/2}n}&\text{if}\ \beta=1/2\\
	Cn^{-(\beta-1/2)}&\text{if} \ \beta\in(1/2,1]
	\end{cases}
	\end{align}
{ for any $n \geq 2$},
	where $X_t$ is { the strong} solution of SDE(\ref{main}) and $X^{(n)}_t$ is
	its Euler{-Maruyama} approximation of step size
	$1/n$. 
{ When $\beta < 1/2$, the existence of a strong solution is
	lost in general~\cite{Barlow} and no numerical scheme is
	available so far.}
	
	In Section~\ref{Scheme}, we {propose a new method of
	approximating SDE (\ref{main})}.
	 In Section~\ref{mainsection}, we provide {  convergence rates of
	 our method under the $\beta$-H\"{o}lder condition 
with $0 < \beta \leq 1$, and some 
	 smoothness condition of diffusion coefficient.
An advantage of our approach is that we approximate the weak solution,
	 which enable us to treat a SDE with no strong solution.
Our scheme is the first to achieve the strong convergence for $0 < \beta
	 < 1/2$, and provides a better convergence rate than
 in \cite{gyongy2002} for $1/2 \leq \beta < 2/3$.}

\section{ Discretization with time change}\label{Scheme}
	Let
	$(\Omega,\mathcal{F},\mathbb{P},\{\mathcal{F}_t\}_{t\geq0})$ be
	a filtered probability space and $\{B_t\}$ be a
	$\{\mathcal{F}_t\}$-Brownian motion. Throughout this paper, we
	consider one-dimensional {SDE(\ref{main}) under the
	following condition.}
\begin{cond}\label{condbdd}
	There are positive constants $C_1,C_2$ such that
 $C_1\leq\sigma(t,x)\leq C_2$ for { all} $(t,x)\in[0,\infty)\times\mathbb{R}$.
\end{cond}
{ Our method is based on the following theorem from
\cite{watanabe}}.

\begin{theo}\label{maintheo}
	Suppose we are given a one dimensional
 $\{\mathcal{F}_t\}$-Brownian motion $b=\{b_t\}$ and an
 $\mathcal{F}_0$-measurable random variable $X_0$. Define $\xi=\{\xi_t\}$ by
 $\xi_t=X_0+b_t$. If $\sigma(t,x)$ satisfies Condition~\ref{condbdd} and
 there exists a process of time change $\varphi$ such that
	\begin{equation}
	\varphi(t)=\int_{0}^{t}\frac{\mathrm{d}s}{\sigma^{2}(\varphi(s),\xi(s))}\label{4}
	\end{equation}
	holds and if such a $\varphi$ is unique, (i.e., $\psi$ is another process of time change satisfying (\ref{4}), then $\varphi_t\equiv\psi_t$ a.s.), then the solution of (\ref{main}) with initial value $X_0$ exists and is unique. Moreover, if we denote $\tau(t):=\varphi^{-1}(t)$, inverse function of $t\mapsto\varphi(t)$, then the solution is given by $X_t=\xi_{\tau(t)}$.
\end{theo}
\begin{rema}
A sufficient condition for the ODE (\ref{4}) to be well-posed 
is that $\sigma(y,x)$ is locally Lipschitz continuous in $y$ and
 satisfies the inequality
	\begin{equation}
		\|\sigma^{-2}(y,x)\|\leq a(x)|y|+b(x)\label{contibbd}
	\end{equation}
	for all $y\in[0,\infty)$ and $x\in\mathbb{R}$, where $a(x)$ and
 $b(x)$ are some continuous non-negative functions of $x$;
see \cite{grigorian}.
In our setting,
the local Lipschitz continuity of $\sigma(t,x)$ in
 $t$ is sufficient because the condition (\ref{contibbd}) follows from 
the boundedness of $\sigma^{-2}(t,x)$.
\end{rema}

The main goal of this paper is to build a numerical approximation of
solution $\{X_t\}$ of the SDE (\ref{main}) using
Theorem~\ref{maintheo}. In order to approximate this time-changed
Brownian motion, we will first make an approximation of Brownian motion
$\{\xi_t\}$ by $\{\xi^{(n)}_t\}$ which is { a} linear interpolation of a
random walk generated by normal distributed random variables, that is
\begin{equation}
\xi^{(n)}_{t}:=\xi_{\lfloor nt\rfloor/n}+(t-\frac{\lfloor nt\rfloor}{n})(\xi_{(\lfloor nt\rfloor+1)/n}-\xi_{\lfloor nt \rfloor/n})\label{discBM}
\end{equation}
where $(\xi_{(\lfloor nt\rfloor+1)/n}-\xi_{\lfloor nt \rfloor/n})\sim\mathcal{N}(0,1/n)$.
Second, we approximate $\{\varphi(t)\}$ by $\varphi_n(t)$, Euler Method for ordinary differential equation,i.e.,
\begin{equation}
\varphi_n(t)=\varphi_n(\frac{k}{n})+(t-\frac{k}{n})\frac{1}{\sigma^2\left(\varphi_n(k/n),\xi_{k/n}\right)},\hspace{10pt}t\in(\frac{k}{n},\frac{k+1}{n}] \label{phi_inter}
\end{equation}
Third, we make the inverse function $\tau_n(t)$ of $t\longmapsto\varphi_n(t)$ by
\begin{equation}
\tau_n(t)=\frac{k}{n}+\frac{t-\varphi_n(\frac{k}{n})}{\varphi_n(\frac{k+1}{n})-\varphi_n(\frac{k}{n})}\frac{1}{n} \label{inverphi_n}
\end{equation}
where $t\in\left[\varphi_n(\frac{k}{n}),\varphi_n(\frac{k+1}{n})\right)$ . We can easily make sure that $\tau_n(t)$ is inverse function of $\varphi_n(t)$ by its definition. 

The concrete algorithm of this method is given as below.
\begin{itemize}
	\item[STEP1] Make $\xi_{t_j},\ j=0,1,2,\cdots$ by normal distributed random sequence and compute $\varphi_n(t_j)$ for these $j$s. 
	\item[STEP2] At the first instant $\varphi_n(t_j)$ passes over $t$, calculate $\tau_n(t)$ using the formula(\ref{inverphi_n}) where we select $k$ in (\ref{inverphi_n}) as $k=j-1$.
	\item[STEP3] Using $\tau_n(t)$ and (\ref{discBM}), calculate $\xi^{(n)}_{\tau_n(t)}$. 
\end{itemize}
Thus we can obtain a path of $\xi^{(n)}_{\tau_n(t)}$ using [STEP1]-[STEP3].
The main result of this paper is discretization error of $\{\xi^{(n)}_{\tau_n(t)}\}$ in the sense of $L^p$ under H\"{o}lder condition of $\sigma(t,s)$, which will be provided in next section.

\section{{Rate of convergence}}
\label{mainsection}
In this section, we provide convergence {rates of our
approximation scheme}. Theorem~\ref{mainresult} declares that under the
{ $\beta$-H\"older continuity of $\sigma(t,x)$ with $\beta \in (0,1]$}
our numerical approximation converges towards the
exact solution in the sense of $L^p$ uniformly and the convergence rate
is $n^{-\alpha^2\beta}$, {where $\alpha$ is an arbitrary value smaller
than $1/2$.}
Theorem~\ref{theosmooth} provides a more precise
convergence rate {$n^{-\alpha}$} when $\sigma$ is sufficiently smooth.

\begin{theo}\label{mainresult}
	Let $\sigma(t,x)$ { satisfy Condition~\ref{condbdd}}
 and suppose that there exist {constants} $C_\beta>0$ and $L_T>0$ such
 that for $s,t\leq T$,
	\begin{align}
		|\sigma(s,x)-\sigma(t,y)|\leq L_T
	 |s-t|+C_\beta|x-y|^\beta.
\label{condthm2}
	\end{align}
	Let $\xi_t,\xi^{(n)}_t,\tau(t),\tau_n(t)$ be defined in the
 previous section. Then, for {any $T>0$, $p\geq 1$ and $\alpha \in [0,1/2)$}, there exists a positive constant
 $\tilde{K}_T$ such that
	\begin{align}
		\left\{E\left[\sup_{t\leq T}\left|\xi^{(n)}_{\tau_n(t)}-\xi_{\tau(t)}\right|^p\right]\right\}^{1/p}\leq \tilde{K}_{T}\  n^{-\alpha^2\beta}.\label{mainrate}
	\end{align}
\end{theo}

We will use the following lemma that is an immediate consequence of
Theorem~(2.1) in \cite{RandY}.
\begin{lemm}\label{reguBM}
	Let $\{\xi_t\}$ be Brownian motion and denote
	\begin{equation}
		H_{\alpha,T}:=\sup_{\substack{s\neq t\\s,t\leq T}}\frac{|\xi_t-\xi_s|}{|t-s|^\alpha}.\label{holderBM}
	\end{equation}
	Then the function $T\mapsto H_{\alpha,T}$ is increasing and 
	\begin{equation}
		E[(H_{\alpha,T})^\gamma]<\infty\nonumber
	\end{equation}
	for any $\alpha\in[0,1/2)$ and $\gamma>0$.
\end{lemm}

\begin{lemm}\label{timebdd}
	Let {$\sigma(t,x)$ satisfy Condition~\ref{condbdd}} and $\varphi(t),\varphi_n(t)$ be defined as (\ref{4}) and (\ref{phi_inter}). Then for each $\gamma>0$ and $T>0$,
	\begin{align}
	\sup_{t\leq T}|\tau_n(t)-\tau(t)|\leq C_2^2\sup_{t\leq C_2^2T}|\varphi_n(t)-\varphi(t)|\nonumber
	\end{align}
\end{lemm}
\begin{proof}
	It follows from Condition~\ref{condbdd} that 
	\begin{align}
	\varphi(t)&\geq\int_{0}^{t}\frac{ds}{C_2^2}=C_2^{-2}t\nonumber,\\
	\varphi_n(t)&\geq\sum_{j=0}^{k-1}\frac{1}{C_2^2}\frac{1}{n}+C_2^{-2}(t-\frac{k}{n})=C_2^{-2}t\nonumber
	\end{align}
	for $t\in[\frac{k}{n},\frac{k+1}{n})$. Therefore, due to the continuity and the strictly increasing property of $\tau(t),\,\tau_n(t)$ and the bounded property of $\varphi(t), \varphi_n(t)$, we get
	\begin{align}
	&\sup_{t\leq T}\left|(\varphi_n^{-1}(t)-\varphi^{-1}(t)\right|\nonumber\\
	\leq&\sup_{t\leq C_2^2 T}\left|(\varphi_n^{-1}(\varphi_n(t))-\varphi^{-1}(\varphi_n(t))\right|\nonumber\\
	=&\sup_{t\leq C_2^2 T}\left|t-\varphi^{-1}(\varphi_n(t))\right|\nonumber\\
	=&\sup_{t\leq C_2^2 T}\left|\varphi^{-1}(\varphi(t))-\varphi^{-1}(\varphi_n(t))\right|\nonumber\\
	\leq&C_2^2\sup_{t\leq C_2^2 T}\left|\varphi(t)-\varphi_n(t)\right|\label{suphi}\nonumber.
	\end{align}
\end{proof}
\begin{proof}[{\bf Proof of Theorem~\ref{mainresult}}]
	First, from Minkowski's inequality, we have
	\begin{align}
		&\left\{ \mathbb{E}\left[\sup_{t\leq T}\left|\xi^{(n)}_{\tau_n(t)} - \xi_{\tau(t)}\right|^p\right] \right\}^{1/p}\nonumber\\
		\leq&  \left\{\mathbb{E}\left[\sup_{t\leq T}\left|\xi^{(n)}_{\tau_n(t)} - \xi_{\frac{\lfloor n\tau_n(t
				)\rfloor}{n}}\right|^p\right]
		\right\}^{1/p}
		+ \left\{ \mathbb{E}\left[\sup_{t\leq T}\left|\xi_{\frac{\lfloor n\tau_n(t
				)\rfloor}{n}} -
	 \xi_{\tau(t)}\right|^p\right] \right\}^{1/p}\nonumber,
	\end{align}
	where $\lfloor t\rfloor$ is the largest integer less than
 $t$. Since $\xi^{(n)}_t$ is { the} interpolation of
 { the} sequence $\{\xi_{j/n}\}_{j=0,1,2,\cdots}$, it follows that
	\begin{equation}
		\left|\xi^{(n)}_{t}-\xi_{\frac{\lfloor nt\rfloor}{n}}\right|\leq\left|\xi^{(n)}_{\frac{\lfloor nt\rfloor +1}{n}}-\xi_{\frac{\lfloor nt\rfloor}{n}}\right|=\left|\xi_{\frac{\lfloor nt\rfloor +1}{n}}-\xi_{\frac{\lfloor nt\rfloor}{n}}\right|\nonumber
	\end{equation}
	Therefore, using Minkowski's inequality again, we obtain
	\begin{align}
	&\left\{ E\left[\sup_{t\leq T}\left|\xi^{(n)}_{\tau_n(t)} - \xi_{\tau(t)}\right|^p\right] \right\}^{1/p}\nonumber\\
	\leq&  \left\{E\left[\sup_{t\leq T}\left|\xi_{\frac{\lfloor n\tau_n(t
			)\rfloor+1}{n}} - \xi_{\frac{\lfloor n\tau_n(t
			)\rfloor}{n}}\right|^p\right]
	\right\}^{1/p}\label{term1}\\
	&+ \left\{ E\left[\sup_{t\leq T}\left|\xi_{\frac{\lfloor n\tau_n(t
			)\rfloor}{n}} - \xi_{\tau_n(t)}\right|^p\right] \right\}^{1/p} \label{term2}\\
	&+\left\{ E\left[\sup_{t\leq T}\left|\xi_{\tau_n(t)} - \xi_{\tau(t)}\right|^p\right] \right\}^{1/p}\label{term3}.
	\end{align}
	Let us provide the desired conclusion by estimating convergence
 rate of (\ref{term1})-(\ref{term3}) at $n\rightarrow\infty$. Define
 $H_{\alpha,T}$ as (\ref{holderBM}) for
 $(\alpha,T)\in(0,1/2)\times[0,\infty)$ and set
 $T':=\max\{T,C_2^2T+1/n\}$ ,  $\tilde{H}:=H_{\alpha,T'}\ (\geq
 H_{\alpha,T})$. Because of the monotone increasing property of $H_t$
 with respect to $t$, Lemma~\ref{reguBM} implies
	\begin{align}
	\left\{E\left[\sup_{t\leq
	 T}\left|\xi_{(\lfloor n \tau_n(t
		)\rfloor+1) n^{-1}} - \xi_{\lfloor n\tau_n(t
		)\rfloor n^{-1}}\right|^p\right]\right\}^{1/p}
	\leq& \left\{E\left[\tilde{H}^p \right]\right\}^{1/p}n^{-\alpha}\label{righterror}\\
	\left\{ E\left[\sup_{t\leq
	 T}\left|\xi_{\lfloor n\tau_n(t
		)\rfloor n^{-1}} - \xi_{\tau_n(t)}\right|^p\right] \right\}^{1/p} \leq& \left\{E\left[\tilde{H}^p \right]\right\}^{1/p}n^{-\alpha}\label{lefterror}
	\end{align}
	Now we have the rate of convergence of { the terms}
 (\ref{term1}) and (\ref{term2}). It remains to prove that the
 convergence rate of the term (\ref{term3}) is
 $n^{-\alpha^2\beta}$. From Lemma~\ref{reguBM}, Lemma~\ref{timebdd} and
 H\"{o}lder's inequality, we have
	\begin{align}
	\left\{E\left[\sup_{t\leq T}\left|\xi_{\varphi_n^{-1}(t)} - \xi_{\varphi^{-1}(t)}\right|^p \right]\right\}^{1/p} \leq& \left\{E\left[\tilde{H}^p C_2^{2p\alpha}\sup_{t\leq C_2^2 T}\left|\varphi(t)-\varphi_n(t)\right|^{p\alpha} \right]\right\}^{1/p}\nonumber\\
	\leq&\left\{E\left[\tilde{H}^{2p}
	\right]\right\}^{1/2p} C_2^{2p\alpha}
	\left\{E\left[\sup_{t\leq T'}\left|\varphi(t)-\varphi_n(t)\right|^{2p\alpha} \right]\right\}^{1/2p}\label{term4}
	\end{align}
		We provide the convergence rate of (\ref{term4}) by estimating the error function $e_n(t):=\varphi^{(n)}(t)-\varphi(t)$. For positive number $h$, define a function $\psi_h:[0,\infty)\rightarrow\mathbb{R}$ as
	\begin{align}
	\psi_h(t):=\frac{1}{h}\int_{t}^{t+h}\left(\sigma^{-2}(\varphi(s),\xi_s)-\sigma^{-2}(\varphi(t),\xi_{t})\right)ds\nonumber
	\end{align}
	From Lemma~\ref{reguBM} and the condition (\ref{condthm2}), for $t\leq T'-h$ and we obtain
	\begin{align}
	|\psi_h(t)|\leq& \frac{1}{h}\left|\int_{t}^{t+h}\left\{\sigma^{-2}(\varphi(s),\xi_s)-\sigma^{-2}(\varphi(t),\xi_{t})\right\}ds\right|\nonumber\\
	\leq&\frac{1}{h}\int_{t}^{t+h}2C_1^{-3}L_{T'}|\varphi(s)-\varphi(t)|+C_\beta|\xi_t-\xi_s|^\beta ds\nonumber\\
	\leq&\frac{1}{h}\int_{t}^{t+h}2C_1^{-3}(L_{T'}+C_\beta)\left|\int_{t}^{s}|C_1^{-2}|du+|\xi_t-\xi_s|^\beta\right|ds\nonumber\\
	\leq&\frac{1}{h}\int_{t}^{t+h}2C_1^{-3}(L_{T'}+C_\beta)\left|\int_{t}^{s}|C_1^{-2}|du+\tilde{H}|t-s|^{\alpha\beta}\right|ds\nonumber.
	\end{align}
	From Lemma~\ref{reguBM}, there is a random variable $R$
 depending on $T'$ which has moments { of any} order and satisfies that
	\begin{align}
	|\psi_h(t)|\leq\frac{1}{h}\int_{t}^{t+h}Rh^{\alpha\beta} ds=Rh^{\alpha\beta}\nonumber.
	\end{align}
	On the other hand, from the definition of $\varphi(t)$, 
	\begin{align}
	\varphi(t)=\varphi(s)+h\sigma(\varphi(s),\xi_{s})+h\psi_h(s),\hspace{10pt}t>s\nonumber
	\end{align}
	where $h=t-s$. Then for $t\in[t_i,t_{i+1}]$,
	\begin{align}
	e_n(t)=&e_n(t_i)+(t-t_i)\{\sigma^{-2}(\varphi^{(n)}(t_i),\xi_{t_i})-\sigma^{-2}(\varphi(t_i),\xi_{t_i})\}+(t-t_i)\psi_{t-t_i}(t_i)\nonumber
	\end{align}
	and by { the} Lipschitz continuity of $\sigma(t,x)$ over with respect to $t$,
	\begin{align}
	|e_n(t)|\leq&|e_n(t_i)|+(t-t_i)L_{T'}|e_n({t_i})|+(t-t_i)|\psi_{t-t_i}(t_i)|\nonumber\\
	\leq&(1+hL_{T'})|e_n({ t_i})|+Rh^{\alpha\beta+1}.\nonumber
	\end{align}
	Repeating this calculus and using the fact that $1+L_{T'}h<\mathrm{e}^{L_{T'}h}$, we have
	\begin{align}
	{\sup_{s\leq t}|e_n(s)|}\leq\frac{Rh^{\alpha\beta}}{L_{T'}}\{(1+hL_{T'})^{i+1}-1\}\leq\frac{Rh^{\alpha\beta}}{L_{T'}}\{\mathrm{e}^{L_{T'}(T'+1)}-1\}\nonumber	\end{align}
	Because of the integrable property of the random variable $R$ and {\it Cauchy-Schwartz's inequality}, there exists a positive number $K$ depending on $T'$ such that
	\begin{align}
	\left\{E\left[\sup_{t\leq  T'}\left|\varphi(t)-\varphi_n(t)\right|^{2p\alpha}\right]\right\}^{1/2p}\leq&K\left(\frac{1}{n}\right)^{\alpha^2\beta}\nonumber
	\end{align}
	Now we complete the proof 
	
\end{proof}
\begin{rema}
	We now have that our approximation converges to the solution of
 (\ref{main}) and the rate of convergence is $n^{-\alpha^2\beta}$.
	Let us compare our result with {(\ref{p1}) by}  Gy\"{o}ngy and
 R\'{a}sonyi~\cite{gyongy2002}.
When $1/2<\beta<2/3$, it is easily seen that
 we can take a number $\alpha\in(0,1/2)$ sufficient closed to $1/2$ for
 $\beta$ such that 
	\begin{align}
	n^{-\alpha^2\beta}<n^{-(\beta-1/2)/p}\nonumber.
	\end{align}
{ Therefore our method enjoys a better estimate of the
 convergence rate than that of  the Euler-Maruyama scheme for $\beta \in
 (1/2,2/3)$. For $\beta \in (0,1/2]$, the convergence of the
 Euler-Maruyama approximation is not known. 
For $\beta \geq 2/3$, 
(\ref{p1}) provides a better rate, while Theorem~\ref{theosmooth} below 
implies that the estimated rate in Theorem~\ref{mainresult} is not sharp when
$\sigma$ is smooth.
}
\end{rema}
{We are going to provide a better estimate of the convergence rate of
our scheme when $\sigma$ is smooth. 
Denote by $\mathcal{L}^q$ }the class of
 	 stochastic process $\{X_t\}$ and ${q}\in\mathbb{N}$ such that 
	 \begin{equation}
	 	E[\int_0^t {|X_s|^q}ds]<\infty,\hspace{10pt}0\leq t<\infty,\nonumber
	 \end{equation}
	 and {by} $\sigma_t,\sigma_x,\sigma_{x,x}$ the partial derivatives of $\sigma$ :
	 \begin{align}
	\sigma_t(t,x):=\frac{\partial\sigma}{\partial t}(t,x),
	  \hspace{5pt}\sigma_x(t,x):=\frac{\partial\sigma}{\partial
	  x}(t,x),\hspace{5pt}\sigma_{xx}(t,x):=\frac{\partial^2\sigma}{\partial
	  x^2}(t,x).
	 \end{align}
\begin{theo}\label{theosmooth}
{Suppose that $\sigma: [0,\infty)\times\mathbb{R}\mapsto\mathbb{R}$}
 	belongs to $C^{2,2}$ and satisfies { the} following conditions
 { in addition to Condition~\ref{condbdd}}:
	\begin{itemize}
		\item[(i)] For { any} $T>0$, there is a constant $L_T>0$ such that 
		\begin{align}
			|\sigma(s,x)-\sigma(t,x)|\leq L_T|s-t|,\hspace{10pt}\forall x\in\mathbb{R},\ \forall s,t\in[0,T].
		\end{align}
		\item[(ii)] There exists some positive constants $C_3, C_4$ such that
		\begin{align}
	{	|\sigma_{xx}(t,x)| + | \sigma_t(t,x)|\leq C_3
		 \exp\{C_4(t+|x|)\}
,\hspace{10pt}\forall x\in\mathbb{R},\ \forall t\in[0,T].}
		\end{align} 
	\end{itemize}
	 Then for all $T>0$, $\alpha\in(0,1/2)$ and {$p \geq 1$}, there exists some constant $K_T>0$ such that
	\begin{equation}
		\left\{E\left[\sup_{t\leq T}\left|\xi^{(n)}_{\tau_n(t)}-\xi_{\tau(t)}\right|^p\right]\right\}^{1/p}\leq K_T n^{-\alpha}\nonumber
	\end{equation}
\end{theo}
\begin{rema}\label{ln}
{	Under the condition {\it (ii)} of
 Theorem~\ref{theosmooth}, since
 $\varphi(t)$ is bounded,
 $\sigma_t(\varphi(t),\xi_t)$, $\sigma_x(\varphi(t),\xi_t)$ and
 $\sigma_{x,x}(\varphi(t),\xi_t)$ belong to $\mathcal{L}^q$ for any $q
 \in \mathbb{N}$ .}
\end{rema}
\begin{proof}[{\bf Proof of Theorem~\ref{theosmooth}}]
{ Under the assumptions of this theorem,
 (\ref{term1})-(\ref{term4}) remain to hold. Therefore, it only} 
remains to estimate {the} convergence rate of (\ref{term4}). More
 precisely, it remains to {prove} that for $\alpha\in(0,1/2)$ and $T>0$
 there exists a constant $K_T>0$ such that
	\begin{equation}
{\left\{E\left[\sup_{t\leq
		  T'}\left|\varphi(t)-\varphi_n(t)\right|^{2p\alpha}
		 \right]\right\}^{1/2p}}
\leq K_Tn^{-\alpha}\nonumber.
	\end{equation}

{We will denote by $C$ a generic constant which depends on
 $p$, $\alpha$
 and $T$, and may change line by line.}
Let us write $X_t:=\sigma^{-2}(\varphi(t),\xi_t)$. Since $\sigma(t,x)$
 belongs to $C^{2,2}$, $X_t$ is { a} semimartingale and can
 be written { as}
	\begin{equation}
		X_t=X_0+M_t+B_t\nonumber,
	\end{equation}
	where 
	\begin{equation}
		X_0:=\sigma^{-2}(0,\xi_0),\hspace{10pt}M_t:=\int_0^t\gamma_s db_s,\hspace{10pt}B_t:=\int_0^t\delta_sds.\label{sigmasemi}
	\end{equation}
{	and $\{\gamma_t\}$ and $\{\delta_t\}$ are in
 $\mathcal{L}^q$ for any $q\in\mathbb{N}$.
 
{Note that}
	\begin{align}
	& \varphi_n(t)-\varphi(t) \\ &= 
\int_0^t\sigma^{-2}(\varphi_n(\frac{\lfloor
	 ns \rfloor}{n}),\xi_{\frac{\lfloor
	 ns\rfloor}{n}})ds  - \int_0^t X_sds 
\\
&=\int_{0}^{t}\left\{X_{\frac{\lfloor
	 ns\rfloor}{n}}-X_s\right\}ds+\int_0^t\left\{\sigma^{-2}(\varphi_n(\frac{\lfloor
	 ns \rfloor}{n}),\xi_{\frac{\lfloor
	 ns\rfloor}{n}})-\sigma^{-2}(\varphi(\frac{\lfloor
	 ns\rfloor}{n}),\xi_{\frac{\lfloor ns \rfloor}{n}})\right\}ds
	\end{align} }
	Since $\sigma^{-2}(t,x)$ is locally Lipschitz continuous, 
{ 
	\begin{align}
		E\left[\sup_{s\leq
	 t}\left|\varphi_n(s)-\varphi(s)\right|^{2p\alpha}\right]\leq
	  C &E\left[\sup_{s\leq
	 t}\left|\int_{0}^{s}X_{\frac{\lfloor
	 nu\rfloor}{n}}-X_udu\right|^{2p\alpha}\right] \\ &+C
\int_{0}^{t}E\left[\sup_{u\leq
	 s}\left|\varphi_n(u)-\varphi(u)\right|^{2p\alpha}\right]ds
	\end{align}}
	for $t\leq T'$. Then by {\it Gronwall's lemma,} we get
	\begin{align}
		E\left[\sup_{s\leq
	 t}\left|\varphi_n(s)-\varphi(s)\right|^{2p\alpha}\right]\leq C
	 E\left[\sup_{t\leq
	 T'}\left|\int_{0}^{t}X_{\frac{\lfloor
	 ns\rfloor}{n}}-X_sds\right|^{2p\alpha}\right].
	\end{align}

	Using {\it by parts formula} for $tX_t$ and $sX_s$ ($s<t$),
	\begin{equation}
		\int_{s}^{t}\left(X_u-X_{s}\right) du=\int_{s}^{t}(t-u) \ dX_u,\hspace{10pt}a.s.\nonumber
	\end{equation}
	and { so we obtain 
	\begin{align}
	 E\left[\sup_{t\leq
	 T^\prime}\left|\int_{0}^{t}X_{\frac{\lfloor
	 ns\rfloor}{n}}-X_sds\right|^{2p\alpha}\right] \leq &
	 C E\left[\sup_{\tau\leq
	 T^\prime}\left|\int_{0}^{\tau}\left(\frac{\lfloor
	 nt\rfloor+1}{n}-t\right)
	 dX_t\right|^{2p\alpha}\right]\\
 & +CE\left[\sup_{\tau\leq
	 T}\left|\int_{0}^{\tau}\left(\frac{\lfloor
	 nt\rfloor+1}{n}\wedge\tau-\frac{\lfloor nt\rfloor+1}{n}\right)
	 dX_t\right|^{2p\alpha}\right]\label{supito}.
	\end{align}}
	Because of the fact that $\frac{\lfloor nt\rfloor+1}{n}\wedge\tau-\frac{\lfloor nt\rfloor+1}{n}=0$ for $t<\frac{\lfloor n\tau \rfloor}{n}$, the definition $X_t=\sigma^{-2}(\varphi(t),\xi_t)$ and Condition~\ref{condbdd}, we can estimate the second term in (\ref{supito}) as follows.
		\begin{align}
		&E\left[\sup_{\tau\leq
			T}\left|\int_{0}^{\tau}\left(\frac{\lfloor
			nt\rfloor+1}{n}\wedge\tau-\frac{\lfloor nt\rfloor+1}{n}\right)
		dX_t\right|^{2p\alpha}\right]\nonumber\\
		=&E\left[\sup_{\tau\leq T}\left|\int_{\frac{\lfloor n\tau\rfloor}{n}}^{\tau}\left(\tau-\frac{\lfloor n\tau\rfloor+1}{n}\right)dX_t\right|^{2p\alpha}\right]\nonumber\\
		=&E\left[\sup_{\tau\leq T}\left|(\tau-\frac{\lfloor
			n\tau\rfloor+1}{n})(X_{\frac{\lfloor
				n\tau\rfloor}{n}}-X_{\tau})\right|^{2p\alpha}\right]\nonumber\\
			\leq& E\left[\frac{1}{n^{2p\alpha}}\sup_{\tau\leq
			T}\left|X_{\frac{\lfloor
				n\tau\rfloor}{n}}-X_{\tau}\right|^{2p\alpha}\right]\leq\frac{C}{n^{2p\alpha}}\label{cuterm}
		\end{align}
	To estimate the first term, we recall the notation (\ref{sigmasemi}) and obtain
	\begin{align}
		&E\left[\sup_{\tau\leq T}\left|\int_{0}^{\tau}\left(\frac{\lfloor nt\rfloor+1}{n}-t \right)dX_t\right|^{2p\alpha}\right]\nonumber\\
		\leq& {C}\left\{E\left[\sup_{\tau\leq
	 T}\left|\int_{0}^{\tau}\left(\frac{\lfloor
	 nt\rfloor+1}{n}-t\right)\
	 {\delta_tdt}\right|^{2p\alpha}\right]+E\left[\sup_{\tau\leq
	 T}\left|\int_{0}^{\tau}\left(\frac{\lfloor
	 nt\rfloor+1}{n}-t\right)\gamma_tdb_t\right|^{2p\alpha}\right]\right\}\label{semintegral}
	\end{align}
	Recalling Remark~\ref{ln}, the
 $\mathcal{L}^{2p\alpha \vee 1}$
 property of $\delta_t$ implies that 
	\begin{align}
E\left[\sup_{\tau\leq
	 T}\left|\int_{0}^{\tau}\left(\frac{\lfloor
	 nt\rfloor+1}{n}-t\right)\delta_tdt\right|^{2p\alpha}\right]
\leq
	 \frac{1}{n^{2p\alpha}}E\left[\left|\int_0^T|\delta_t|dt\right|^{2p\alpha}\right]
	 \leq
	 \frac{C}{n^{2p\alpha}}\label{drifterm}.
	\end{align}
Here we have used that
\begin{equation*}
 E\left[\left|\int_0^T|\delta_t|dt\right|^{2p\alpha}\right]
\leq 
T^{2p\alpha -1 }E\left[\int_0^T|\delta_t|^{2p\alpha}dt\right] < \infty
\end{equation*}
if $2p\alpha \geq  1$  and
 \begin{equation*}
 E\left[\left|\int_0^T|\delta_t|dt\right|^{2p\alpha}\right]
\leq 
E\left[\int_0^T|\delta_t|dt\right]^{2p\alpha} < \infty
\end{equation*}
if $2p\alpha <  1$.

	On the other hand, using { \it
 { the Burkholder-Davis-Gundy inequality}} and the $\mathcal{L}^{2p\alpha}$ property of
 {$\gamma_t$}, we obtain that for the second term of
 (\ref{semintegral}), 
	\begin{align}
		E\left[\sup_{\tau\leq
	 T}\left|\int_{0}^{\tau}\left(\frac{\lfloor
	 nt\rfloor+1}{n}-t\right)\gamma_tdb_t\right|^{2p\alpha}\right]\leq
{C}E\left[\left|\int_{0}^{T}\left|\frac{\lfloor
	 nt \rfloor+1}{n}-t\right|^2 \gamma_t^2dt\right|^{p\alpha}\right]\leq
	 \frac{C}{n^{2p\alpha}}\label{marterm}.
	\end{align}
	From (\ref{supito}) and (\ref{cuterm})-(\ref{marterm}), it
 follows that
\begin{equation*}
	 E\left[\sup_{t\leq
	 T'}\left|\int_{0}^{t}X_{\frac{\lfloor
	 ns\rfloor}{n}}-X_sds\right|^{2p\alpha}\right] 
\leq \frac{C}{n^{2p\alpha}},
\end{equation*}
which concludes the proof.

\end{proof}

\end{document}